\documentclass[a4paper,english]{article}
\usepackage[a4paper]{geometry}
\usepackage{amsmath,amsthm,amssymb,mathtools}
\usepackage{enumerate,csquotes,verbatim}
\usepackage{setspace}
\usepackage[numbers]{natbib}
\usepackage{hyperref}
\usepackage{color,xcolor}
\usepackage{tikz}
\usetikzlibrary{arrows}
\usepackage{multirow}
\usepackage{graphicx}
\usepackage{tabularx}
\usepackage{float} 		
\restylefloat{table}    

\newtheorem{thm}{Theorem}
\newtheorem{lem}[thm]{Lemma}

\newtheorem{prop}[thm]{Proposition}

\theoremstyle{definition}

\title{A strategy for Isolator in the Toucher-Isolator game on trees}

\author{Sopon Boriboon\thanks{\,Department of Mathematics and Computer Science, Faculty of Science, Chulalongkorn University, Bangkok 10330, Thailand; \texttt{soponboriboon@gmail.com}.}
  \and Teeradej Kittipassorn\thanks{\,Department of Mathematics and Computer Science, Faculty of Science, Chulalongkorn University, Bangkok 10330, Thailand; \texttt{teeradej.k@chula.ac.th}.}}

\begin{document}
	
	\maketitle
	
	\begin{abstract}
	In the Toucher-Isolator game, introduced recently by Dowden, Kang, Mikala\v{c}ki and Stojakovi\'{c}, Toucher and Isolator alternately claim an edge from a graph such that Toucher aims to touch as many vertices as possible, while Isolator aims to isolate as many vertices as possible, where Toucher plays first. Among trees with $n$ vertices, they showed that the star is the best choice for Isolator and they asked for the most suitable tree for Toucher. Later, R\"{a}ty showed that the answer is the path with $n$ vertices. We give a simple alternative proof of this result. The method to determine where Isolator should play is by breaking down the gains and losses in each move of both players. 
	\end{abstract}
	
	\section{Introduction}\label{section-introduction}
	
	A Maker-Breaker game, introduced by Erd\H{o}s and Selfridge~\cite{erdos1973combinatorial} in 1973, is a positional game played on the complete graph $K_n$ with $n$ vertices, by two players: Maker and Breaker, who alternately claim an edge from the (remaining) graph, where Maker plays first. Maker wins if she can build a particular structure (e.g., a clique~\cite{MR2788689,MR2854626},  a perfect matching~\cite{MR2467817,MR3870440} or a Hamiltonian cycle ~\cite{MR2467817,MR2726601}) from her claimed edges, while Breaker wins if he can prevent this. There are several variants of Maker-Breaker games, many of which are studied recently (see~\cite{espig2015walker,forcan2019walkermaker,gledel2019maker,MR3963857}).
		
	The \textit{Toucher-Isolator} game, introduced by Dowden, Kang, Mikala\v{c}ki and Stojakovi\'{c}~\cite{MR4025410} in 2019, is a quantitative version of a Maker-Breaker game played on a finite graph by two players: \textit{Toucher} and \textit{Isolator}, who alternately claim an edge from the (remaining) graph, where Toucher plays first. A vertex is \emph{touched} if it is incident to at least one edge claimed by Toucher, and a vertex is \emph{untouched} if all edges incident to it are claimed by Isolator. The \textit{score} of the game is the number of untouched vertices at the end of the game when all edges  have  been claimed. Toucher aims at minimizing the score, while Isolator aims at maximizing the score. For a graph $G$, let $u(G)$ be the score of the game on $G$ when both players play optimally.  
	
	The above mentioned authors gave general upper and lower bounds for $u(G)$, leaving the asymptotic behavior of $u(C_n)$ and $u(P_n)$ as the most interesting unsolved cases, where $C_n$ is a cycle with $n$ vertices  and $P_n$ is a path with $n$ vertices. Later in 2019, R\"{a}ty~\cite{raty2019achievement} determined the exact values of $u(C_n)$ and $u(P_n)$, showing that
	\begin{center}
		$u(C_n)=\left\lfloor\dfrac{n+1}{5}\right\rfloor$ \quad and \quad $u(P_n)=\left\lfloor\dfrac{n+3}{5}\right\rfloor.$
	\end{center}
	Moreover, the first set of authors showed that for any tree $T$ with $n\geq3$ vertices, 
	\begin{center}
		$\dfrac{n+2}{8}\leq u(T)\leq\dfrac{n-1}{2}$,
	\end{center}
	where the upper bound is tight when $T$ is a star, but the only tight example they found for the lower bound is a path with six vertices. Therefore, they asked whether there is an infinite family of tight examples for lower bound, or it can be improved for large $n$. 
	
	Later in 2020, R\"{a}ty~\cite{raty2020toucher} improved the lower bound for $u(T)$ by showing that the path $P_n$ is the most suitable tree with $n$ vertices for Toucher.
	
	\begin{thm}\label{thm:uT}
		Let $T$ be a tree with $n\geq3$ vertices. Then 
		\begin{center}
			$u(T)\geq\left\lfloor\dfrac{n+3}{5}\right\rfloor$.
		\end{center}
	\end{thm}

	In this paper, we give a simple new proof of this theorem. The argument proceeds as follows. The strategy for Isolator is that he claims an edge which immediately creates an untouched vertex in every move for as long as he can (see Figure~\ref{figure:thm}: left). When no such an edge exists, we modify the graph before the game continues. The edges claimed by Isolator can be deleted as their disappearance does not change the touched/untouched status of any vertex (see Figure~\ref{figure:thm}: middle). Observe that the leaves of the remaining tree are touched otherwise Isolator would have claimed the edge incident to it. Then we delete the edges $e$ claimed by Toucher one by one and, in order to keep the game equivalent to the original game, we replace the edges $u_1v,\dots,u_tv$ sharing a vertex $v$ with $e$ by new edges $u_1v_1,\dots,u_tv_t$ keeping their respective Toucher/Isolator status, where the new vertices  $v_1,\dots,v_t$ are considered touched. The resulting graph is a forest all of whose leaves are considered touch (see Figure~\ref{figure:thm}: right).
		
			Therefore, this motivates us to study the \emph{non-leaf Isolator-Toucher game} on a forest $F$ which is a variant of the Toucher-Isolator game on $F$ where Isolator plays first and the score of the game is the number of untouched vertices, which are not leaves of $F$, at the end of the game. The aim of Toucher is to minimize the score, while the aim of Isolator is to maximize the score. We remark that this game is inspired by the proof of the lower bound for $P_n$ in~\cite{raty2019achievement}.  Our main lemma gives a lower bound for the minimum score $\alpha(m,k,l)$ of the non-leaf Isolator-Toucher game on $F$ when both players play optimally, among all forests $F$ with $m$ edges, $k$ components, and $l$ leaves.
	
	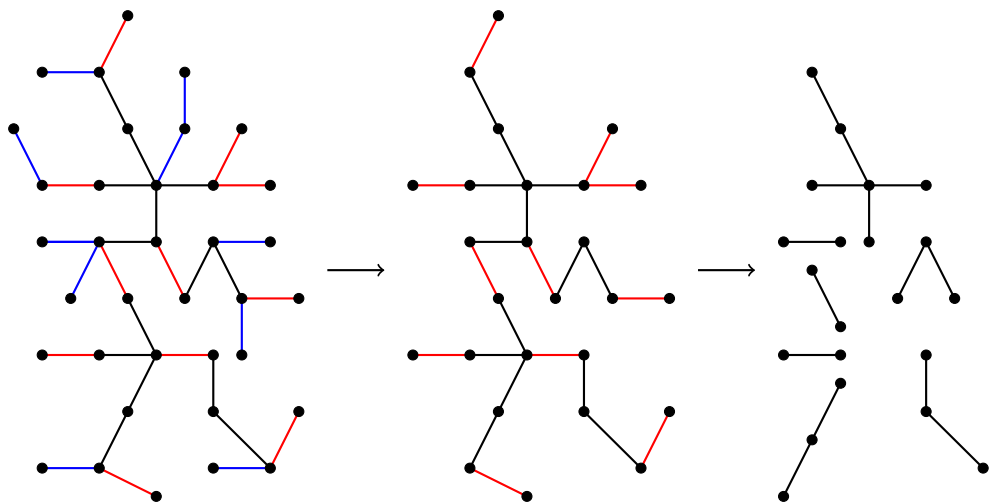
\begin{figure}[H] \centering
		\begin{tikzpicture}[baseline=2ex,scale=0.75]
		\draw[thick,blue] (0.5,7) -- (1,6);
		\draw[thick,blue] (1,1) -- (2,1);
		\draw[thick,blue] (1,5) -- (2,5) -- (1.5,4);
		\draw[thick,blue] (1,8) -- (2,8);
		\draw[thick,blue] (4,1) -- (5,1);
		\draw[thick,blue] (3,6) -- (3.5,7) -- (3.5,8);	
		\draw[thick,blue] (1,5) -- (2,5);
		\draw[thick,blue] (4.5,3) -- (4.5,4);
		\draw[thick,blue] (4,5) -- (5,5);
		\draw[thick,red] (1,3) -- (2,3);
		\draw[thick,red] (1,6) -- (2,6);
		\draw[thick,red] (2,1) -- (3,0.5);
		\draw[thick,red] (2,5) -- (2.5,4);
		\draw[thick,red] (2,8) -- (2.5,9);
		\draw[thick,red] (3,3) -- (4,3);
		\draw[thick,red] (3,5) -- (3.5,4);
		\draw[thick,red] (5,1) -- (5.5,2);
		\draw[thick,red] (4.5,4) -- (5.5,4);
		\draw[thick,red] (4.5,7) -- (4,6) -- (5,6);  
		\draw[thick] (2,1) -- (2.5,2) -- (3,3) -- (2,3);
		\draw[thick] (3,3) -- (2.5,4);
		\draw[thick] (2,5) -- (3,5) -- (3,6) -- (2,6) ;
		\draw[thick] (2,8) -- (2.5,7) -- (3,6) -- (4,6) ;
		\draw[thick] (3.5,4) -- (4,5) -- (4.5,4);	
		\draw[thick] (4,3) -- (4,2) -- (5,1);
		\draw[thick,->] (6,4.5) -- (7,4.5);
		\draw[fill=black] (0.5,7) circle (2.5pt);
		\draw[fill=black] (1,1) circle (2.5pt);
		\draw[fill=black] (1,3) circle (2.5pt);
		\draw[fill=black] (1,5) circle (2.5pt);
		\draw[fill=black] (1,6) circle (2.5pt);
		\draw[fill=black] (1,8) circle (2.5pt);
		\draw[fill=black] (1.5,4) circle (2.5pt);
		\draw[fill=black] (2,1) circle (2.5pt);
		\draw[fill=black] (2,3) circle (2.5pt);
		\draw[fill=black] (2,5) circle (2.5pt);
		\draw[fill=black] (2,6) circle (2.5pt);
		\draw[fill=black] (2,8) circle (2.5pt);
		\draw[fill=black] (2.5,2) circle (2.5pt);
		\draw[fill=black] (2.5,4) circle (2.5pt);
		\draw[fill=black] (2.5,7) circle (2.5pt);
		\draw[fill=black] (2.5,9) circle (2.5pt);
		\draw[fill=black] (3,0.5) circle (2.5pt);
		\draw[fill=black] (3,3) circle (2.5pt);
		\draw[fill=black] (3,5) circle (2.5pt);
		\draw[fill=black] (3,6) circle (2.5pt);
		\draw[fill=black] (3.5,4) circle (2.5pt);
		\draw[fill=black] (3.5,7) circle (2.5pt);
		\draw[fill=black] (3.5,8) circle (2.5pt);
		\draw[fill=black] (4,1) circle (2.5pt);
		\draw[fill=black] (4,2) circle (2.5pt);
		\draw[fill=black] (4,3) circle (2.5pt);
		\draw[fill=black] (4,5) circle (2.5pt);
		\draw[fill=black] (4,6) circle (2.5pt);
		\draw[fill=black] (4.5,3) circle (2.5pt);
		\draw[fill=black] (4.5,4) circle (2.5pt);
		\draw[fill=black] (4.5,7) circle (2.5pt);
		\draw[fill=black] (5,1) circle (2.5pt);
		\draw[fill=black] (5,5) circle (2.5pt);
		\draw[fill=black] (5,6) circle (2.5pt);
		\draw[fill=black] (5.5,2) circle (2.5pt);
		\draw[fill=black] (5.5,4) circle (2.5pt);
		\draw[thick,red] (7.5,3) -- (8.5,3);
		\draw[thick,red] (7.5,6) -- (8.5,6);
		\draw[thick,red] (8.5,1) -- (9.5,0.5);
		\draw[thick,red] (8.5,5) -- (9,4);
		\draw[thick,red] (8.5,8) -- (9,9);
		\draw[thick,red] (9.5,3) -- (10.5,3);
		\draw[thick,red] (9.5,5) -- (10,4);
		\draw[thick,red] (11.5,1) -- (12,2);
		\draw[thick,red] (11,4) -- (12,4);
		\draw[thick,red] (11,7) -- (10.5,6) -- (11.5,6);  
		\draw[thick] (8.5,1) -- (9,2) -- (9.5,3) -- (8.5,3);
		\draw[thick] (9.5,3) -- (9,4);
		\draw[thick] (8.5,5) -- (9.5,5) -- (9.5,6) -- (8.5,6) ;
		\draw[thick] (8.5,8) -- (9,7) -- (9.5,6) -- (10.5,6) ;
		\draw[thick] (10,4) -- (10.5,5) -- (11,4);	
		\draw[thick] (10.5,3) -- (10.5,2) -- (11.5,1);
		\draw[thick,->] (12.5,4.5) -- (13.5,4.5);
		\draw[fill=black] (7.5,3) circle (2.5pt);
		\draw[fill=black] (7.5,6) circle (2.5pt);
		\draw[fill=black] (8.5,1) circle (2.5pt);
		\draw[fill=black] (8.5,3) circle (2.5pt);
		\draw[fill=black] (8.5,5) circle (2.5pt);
		\draw[fill=black] (8.5,6) circle (2.5pt);
		\draw[fill=black] (8.5,8) circle (2.5pt);
		\draw[fill=black] (9,2) circle (2.5pt);
		\draw[fill=black] (9,4) circle (2.5pt);
		\draw[fill=black] (9,7) circle (2.5pt);
		\draw[fill=black] (9,9) circle (2.5pt);
		\draw[fill=black] (9.5,0.5) circle (2.5pt);
		\draw[fill=black] (9.5,3) circle (2.5pt);
		\draw[fill=black] (9.5,5) circle (2.5pt);
		\draw[fill=black] (9.5,6) circle (2.5pt);
		\draw[fill=black] (10,4) circle (2.5pt);
		\draw[fill=black] (10.5,2) circle (2.5pt);
		\draw[fill=black] (10.5,3) circle (2.5pt);
		\draw[fill=black] (10.5,5) circle (2.5pt);
		\draw[fill=black] (10.5,6) circle (2.5pt);
		\draw[fill=black] (11,4) circle (2.5pt);
		\draw[fill=black] (11,7) circle (2.5pt);
		\draw[fill=black] (11.5,1) circle (2.5pt);
		\draw[fill=black] (11.5,6) circle (2.5pt);
		\draw[fill=black] (12,2) circle (2.5pt);
		\draw[fill=black] (12,4) circle (2.5pt);
		\draw[thick] (14,3) -- (15,3);
		\draw[thick] (14,0.5) -- (14.5,1.5) -- (15,2.5);
		\draw[thick] (14.5,4.5) -- (15,3.5);
		\draw[thick] (14,5) -- (15,5);
		\draw[thick] (14.5,8) -- (15,7) -- (15.5,6) -- (15.5,5);	
		\draw[thick] (14.5,6) -- (16.5,6) ;
		\draw[thick] (16.5,3) -- (16.5,2) -- (17.5,1);
		\draw[thick] (16,4) -- (16.5,5) -- (17,4) ;
		\draw[fill=black] (14,3) circle (2.5pt);
		\draw[fill=black] (14,5) circle (2.5pt);
		\draw[fill=black] (14,0.5) circle (2.5pt);
		\draw[fill=black] (14.5,4.5) circle (2.5pt);
		\draw[fill=black] (14.5,6) circle (2.5pt);
		\draw[fill=black] (14.5,8) circle (2.5pt);
		\draw[fill=black] (14.5,1.5) circle (2.5pt);
		\draw[fill=black] (15,3) circle (2.5pt);
		\draw[fill=black] (15,5) circle (2.5pt);
		\draw[fill=black] (15,7) circle (2.5pt);
		\draw[fill=black] (15,2.5) circle (2.5pt);
		\draw[fill=black] (15,3.5) circle (2.5pt);
		\draw[fill=black] (15.5,5) circle (2.5pt);
		\draw[fill=black] (15.5,6) circle (2.5pt);
		\draw[fill=black] (16,4) circle (2.5pt);
		\draw[fill=black] (16.5,2) circle (2.5pt);
		\draw[fill=black] (16.5,3) circle (2.5pt);
		\draw[fill=black] (16.5,5) circle (2.5pt);
		\draw[fill=black] (16.5,6) circle (2.5pt);
		\draw[fill=black] (17.5,1) circle (2.5pt);
		\draw[fill=black] (17,4) circle (2.5pt);
		\end{tikzpicture} 
		\caption{The strategy for Isolator in the Toucher-Isolator game on a tree and the modification of the graph, where the red and blue edges are Toucher and Isolator edges respectively.} \label{figure:thm}
	\end{figure}
	
	\begin{lem} \label{lem:alpha} 
		For non-negative numbers  $m$, $k$ and $l$, 
		\begin{center}
			$\alpha(m,k,l)\geq\left\lfloor\dfrac{m+4k-3l+4}{5}\right\rfloor$. 
		\end{center}
	\end{lem}
	
	The strategy for Isolator in the non-leaf Isolator-Toucher game is that he claims consecutive edges which  immediately creates an untouched vertex in every move except the first one for as long as he can, and then he repeats in a different part of the forest. The key step is to determine which part of the forest is the most profitable for Isolator to play in. We do this by breaking down the gains and losses in each move of both players.
 
	The rest of this paper is organized as follows. Section~\ref{section-the proofs} is devoted to proving Lemma~\ref{lem:alpha} and then applying it to prove Theorem~\ref{thm:uT}. In Section~\ref{section-conclusion}, we give some concluding remarks and mention related interesting questions.
	
	\section{The Proofs}\label{section-the proofs}

	Before proving Lemma~\ref{lem:alpha} and Theorem~\ref{thm:uT}, we give some definitions necessary for the proofs and make observations regarding how to modify the graph after deleting some edges, to keep the game equivalent to the original game, and how much Isolator gains in each move of both players. 
	
	For convenience, we first give some names to vertices and edges in a forest.  A \emph{leaf} is a vertex of degree $1$. A \emph{small vertex} is a vertex of degree $2$. A \emph{big vertex} is a vertex of degree at least~$3$. A \emph{big edge} is an edge incident to a big vertex. A \emph{leaf edge} is an edge incident to a leaf. An \textit{internal vertex} of a subgraph is a vertex  adjacent to no vertex outside the subgraph.
	
	We also give some names to paths in a forest. A \emph{path component} is a  path such that the non-endpoint vertices are internal and both endpoints are leaves. A \emph{branch} is a path such that the non-endpoint vertices are internal and both endpoints are big. A \emph{twig} is a path such that the non-endpoint vertices are internal and one endpoint is a leaf while the other is big.
	
	Finally, we define some game related terms. A \emph{Toucher edge} is an edge claimed by Toucher. An \emph{Isolator edge} is an edge claimed by Isolator. An \textit{Isolator subgraph} is a subgraph whose edges are Isolator edges. An \emph{Isolator path} is an Isolator subgraph which is either a path component, a brach or a twig. A \emph{partially played graph} is a graph where each edge is either a Toucher edge, an Isolator edge or an unclaimed edge. 
	
	Now we say how a partially played graph should be modified after deleting a Toucher edge or an Isolator subgraph, in order to keep the game equivalent to the original game. For a partially played graph $G$ with a Toucher edge $uv$, we define $G\circleddash uv$ to be the partially played graph obtained from $G$ by
	\begin{itemize}
		\item deleting the vertices $u$ and $v$,  
		\item adding new vertices $u_1,\dots,u_{d(u)-1}$ and joining $u_i$ to $u'_i$ where $N(u)\setminus \{v\}=\{u'_1,\dots,u'_{d(u)-1}\}$ such that if $uu'_i$ has been claimed by a player, then we let $u_iu'_i$ be claimed by the same player, 
		\item adding new vertices $v_1,\dots,v_{d(v)-1}$ and joining $v_i$ to $v'_i$ where $N(v)\setminus \{u\}=\{v'_1,\dots,v'_{d(v)-1}\}$ such that if $vv'_i$ has been claimed by a player, then we let $v_iv'_i$ be claimed by the same player,
	\end{itemize}
	 where $N(v)$ denotes the neighborhood of vertex $v$ and $d(v)$ denotes the degree of vertex $v$.
	
	\begin{figure}[H] \centering
		\begin{tikzpicture}[baseline=2ex]
		\draw[thick,red] (2,2) -- (3,2);
		\draw[thick,red] (9,2.5) -- (10,2.5);
		\draw[thick,red] (1,2.5) -- (2,2);
		\draw[thick,red] (5,1.5) -- (4,2);
		\draw[thick,red](13.5,1.5) -- (12.5,2); 
		\draw[thick,blue] (9,1.5) -- (10,1.5);
		\draw[thick,blue] (4,2) -- (5,2);
		\draw[thick,blue] (1,1.5) -- (2,2);
		\draw[thick,blue] (12.5,2) -- (13.5,2);
		\draw[thick] (0,1) -- (1,1.5) -- (0,2);
		\draw[thick] (2,2) -- (1,2);
		\draw[thick] (0,3) -- (1,2.5);
		\draw[thick] (3,2) -- (4,2);
		\draw[thick] (4,2) -- (5,2.5);
		\draw[thick,->] (6,2) -- (7,2);
		\draw[thick] (8,1)  -- (9,1.5) -- (8,2);
		\draw[thick] (9,2) -- (10,2);
		\draw[thick] (9,2.5) -- (8,3);
		\draw[thick] (11.5,2) -- (12.5,2);
		\draw[thick] (12.5,2) -- (13.5,2.5); 
		\draw[fill=black] (0,1) circle (2.5pt);
		\draw[fill=black] (0,2) circle (2.5pt);
		\draw[fill=black] (0,3) circle (2.5pt);
		\draw[fill=black] (1,1.5) circle (2.5pt);
		\draw[fill=black] (1,2) circle (2.5pt);
		\draw[fill=black] (1,2.5) circle (2.5pt);
		\draw[fill=black] (2,2) circle (2.5pt) node[above=2pt] {$u$};
		\draw[fill=black] (3,2) circle (2.5pt) node[above=2pt] {$v$};
		\draw[fill=black] (4,2) circle (2.5pt);
		\draw[fill=black] (5,2) circle (2.5pt);
		\draw[fill=black] (5,1.5) circle (2.5pt);
		\draw[fill=black] (5,2.5) circle (2.5pt);
		\draw[fill] (2.5,0.5)  node[above=2pt] {$G$};
		\draw[fill=black] (8,1) circle (2.5pt);
		\draw[fill=black] (8,2) circle (2.5pt);
		\draw[fill=black] (8,3) circle (2.5pt);
		\draw[fill=black] (9,1.5) circle (2.5pt);
		\draw[fill=black] (9,2) circle (2.5pt);
		\draw[fill=black] (9,2.5) circle (2.5pt);
		\draw[fill=black] (10,1.5) circle (2.5pt) node[right=2pt] {$u_3$};
		\draw[fill=black] (10,2) circle (2.5pt) node[right=2pt] {$u_2$};
		\draw[fill=black] (10,2.5) circle (2.5pt) node[right=2pt] {$u_1$};
		\draw[fill=black] (11.5,2) circle (2.5pt) node[left=2pt] {$v_1$};
		\draw[fill=black] (12.5,2) circle (2.5pt);
		\draw[fill=black] (13.5,1.5) circle (2.5pt);
		\draw[fill=black] (13.5,2) circle (2.5pt);
		\draw[fill=black] (13.5,2.5) circle (2.5pt);
		\draw[fill] (11,0.5)  node[above=2pt] {$G\circleddash uv$};
		\end{tikzpicture}  \caption{The partially played graph $G\circleddash uv$, where the red and blue edges are Toucher and Isolator edges respectively.}
	\end{figure}
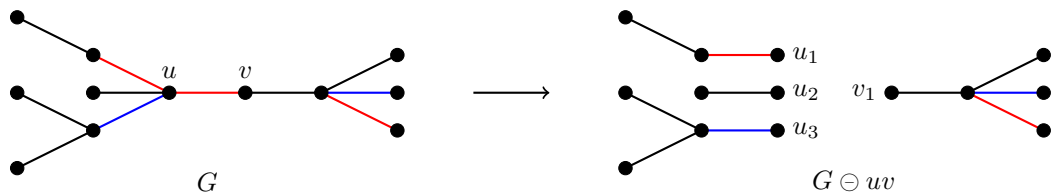
	
	For a partially played graph $G$ with an Isolator subgraph $H$, we define $G\circleddash H$ to be the partially played graph obtained from $G$ by deleting the edges of $H$ and the internal vertices of~$H$.

	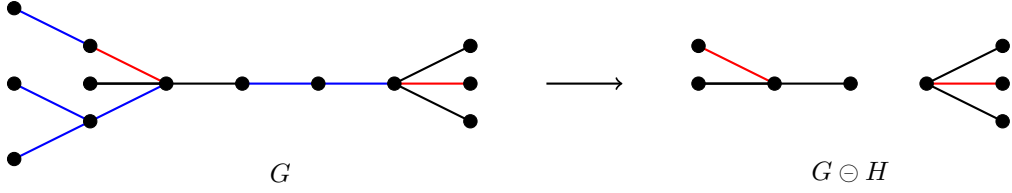
\begin{figure}[H] \centering
		\begin{tikzpicture}[baseline=2ex] 	
		\draw[thick,blue] (0,1) -- (1,1.5) -- (0,2);
		\draw[thick,blue] (0,3) -- (1,2.5);
		\draw[thick,blue] (3,2) -- (4,2) -- (5,2);
		\draw[thick,blue] (2,2) -- (1,1.5);	
		\draw[thick,red] (1,2.5) -- (2,2);	
		\draw[thick,red] (9,2.5) -- (10,2);	
		\draw[thick,red] (12,2) -- (13,2);
		\draw[thick,red] (5,2) --(6,2);
		\draw[thick] (1,2) -- (3,2);
		\draw[thick] (1,2) -- (2,2);
		\draw[thick] (6,1.5) -- (5,2) -- (6,2.5);
		\draw[thick,->] (7,2) -- (8,2);
		\draw[thick] (10,2) -- (9,2);
		\draw[thick] (9,2) -- (10,2) -- (11,2) ;
		\draw[thick] (13,1.5) -- (12,2) -- (13,2.5);
		\draw[fill=black] (0,1) circle (2.5pt);
		\draw[fill=black] (0,2) circle (2.5pt);
		\draw[fill=black] (0,3) circle (2.5pt);
		\draw[fill=black] (1,1.5) circle (2.5pt);
		\draw[fill=black] (1,2) circle (2.5pt);
		\draw[fill=black] (1,2.5) circle (2.5pt);
		\draw[fill=black] (2,2) circle (2.5pt);
		\draw[fill=black] (3,2) circle (2.5pt);
		\draw[fill=black] (4,2) circle (2.5pt);
		\draw[fill=black] (5,2) circle (2.5pt);
		\draw[fill=black] (6,2) circle (2.5pt);
		\draw[fill=black] (6,1.5) circle (2.5pt);
		\draw[fill=black] (6,2.5) circle (2.5pt);
		\draw[fill] (3.5,0.5)  node[above=2pt] {$G$};
		\draw[fill=black] (9,2) circle (2.5pt);
		\draw[fill=black] (9,2.5) circle (2.5pt);
		\draw[fill=black] (10,2) circle (2.5pt);
		\draw[fill=black] (11,2) circle (2.5pt);
		\draw[fill=black] (12,2) circle (2.5pt);
		\draw[fill=black] (13,2) circle (2.5pt);
		\draw[fill=black] (13,1.5) circle (2.5pt);
		\draw[fill=black] (13,2.5) circle (2.5pt);
		\draw[fill] (11,0.5)  node[above=2pt] {$G\circleddash H$};
		\end{tikzpicture} \caption{The partially played graph $G\circleddash H$, where the red and blue edges are Toucher and Isolator edges respectively.}
	\end{figure}
	
	\begin{prop} \label{prop:equi}
		\begin{enumerate}[(i)]
			\item The non-leaf Isolator-Toucher game on a partially played graph $G$ with a Toucher edge $e$ is equivalent to that on $G\circleddash e$.
			\item The Toucher-Isolator and the non-leaf Isolator-Toucher games on a partially played graph $G$ with an Isolator subgraph $H$ with $r$ internal vertices are equivalent to their respective versions on $G\circleddash H$ with an extra score of $r$.
			\item The score of the non-leaf Isolator-Toucher game on a partially played graph $G$ when both players play optimally is equal to that on $G-U$, where $U$ is the set of vertices of path components of length $1$ in $G$.
		\end{enumerate}
		
	\end{prop}
	
	\begin{proof}
		$(i)$ Clearly, there is a bijection between the edges of $G-e$ and $G\circleddash e$. The endpoints of the Toucher edge $e$ in the game on $G$ and the new leaves in the game on $G\circleddash e$ are not counted in the score of each game.
		
		$(ii)$ Clearly, there is a bijection between the edges of $G-E(H)$ and $G\circleddash H$. Deleting an Isolator edge does not change the touch/untouched status of its endpoints. An extra score of $r$ comes from the internal vertices on $H$.
		
		$(iii)$ A player gains nothing by claiming  a path component of length $1$ because its vertices are leaves which are not counted in the score.
	\end{proof}

	Next, in order to determine which part of the forest is the most profitable for Isolator to play in, it is useful to calculate the changes in the number of edges, components and leaves of the forest when deleting a Toucher edge or an Isolator path. Moreover, deleting path components of length $1$ also produces a profit.
	\begin{prop} \label{prop:profit} 
		\begin{enumerate} [(i)]
			\item Let $G$ be a partially played graph which is a forest with $m$ edges, $k$ components and $l$ leaves, and let $uv$ be a Toucher edge in $G$. Suppose $G\circleddash uv$ is a forest with $m+\Delta m$ edges, $k+\Delta k$ components and $l+\Delta l$ leaves. Then the change in $m+4k-3l$ is as in Table~\ref{table:toucher} and the profit $p_T(G,uv)=\Delta(m+4k-3l)+3$ is non-negative.
		\begin{table}[H] \centering
			\resizebox{\textwidth}{!}{	\begin{tabular}{|p{12m}|p{12mm}|c|r|r|r|c|}
					\hline
					\multicolumn{2}{|c|}{Toucher edge $uv$} & \multirow{2}{*}{$\Delta m$} & \multicolumn{1}{c|}{\multirow{2}{*}{$\Delta k$}} & \multicolumn{1}{c|}{\multirow{2}{*}{$\Delta l$}} & \multicolumn{1}{c|}{\multirow{2}{*}{$\Delta(m+4k-3l)$}} & \multicolumn{1}{c|}{\multirow{2}{*}{$p_T(G,uv)$}} \\ \cline{1-2}
					\multicolumn{1}{|c|}{\phantom{text}$u$\phantom{text}} & \multicolumn{1}{c|}{\phantom{text}$v$\phantom{text}}     &   & \multicolumn{1}{c|}{}   & \multicolumn{1}{c|}{}  & \multicolumn{1}{c|}{}     & \multicolumn{1}{c|}{}                   \\ \hline
					\multicolumn{1}{|c|}{small}   & \multicolumn{1}{c|}{small}   & $-1$ & $1$ & $2$ & $-3$&  \phantom{$\geq$ }$0$ \\ \hline	
					\multicolumn{1}{|c|}{small}   &\multicolumn{1}{c|}{big}	  & $-1$ & $d(v)-1$ & $d(v)$ & $d(v)-5\geq-2$  & $\geq1$ \\ \hline
					\multicolumn{1}{|c|}{small}   &\multicolumn{1}{c|}{leaf}	 	  & $-1$ & $0$ & $0$ & $-1$ &  \phantom{$\geq$ }$2$\\ \hline
					\multicolumn{1}{|c|}{big}	  & \multicolumn{1}{c|}{big} & $-1$ & $d(u)+d(v)-3$ & $d(u)+d(v)-2$ &$d(u)+d(v)-7\geq-1$ &  $\geq2$\\ \hline	
					\multicolumn{1}{|c|}{big}	  & \multicolumn{1}{c|}{leaf}	      & $-1$ & $d(u)-2$ & $d(u)-2$ & $d(u)-3\geq\phantom{-}0$ & $\geq3$ \\ \hline
					\multicolumn{1}{|c|}{leaf}	  & \multicolumn{1}{c|}{leaf}		  & $-1$ & $-1$ & $-2$  & $1$ & \phantom{$\geq$ }$4$ \\ \hline		 
			\end{tabular}}
				\caption{The profit of deleting a Toucher edge.}\label{table:toucher}
			\end{table}
			\item Let $G$ be a partially played graph which is a forest with $m$ edges, $k$ components and $l$ leaves, and let $P$ be an Isolator path of length $r+1$ in $G$. Suppose $G\circleddash P$ is a forest with $m+\Delta m$ edges, $k+\Delta k$ components and $l+\Delta l$ leaves. Then the change in $m+4k-3l$ is as in Table~\ref{table:isolator} and the profit $p_I(G,P)=\Delta(m+4k-3l)+r-1$ is non-negative.
			\begin{table}[H]   \centering
			\begin{tabular}{|c|c|c|r|r|c|c|}
				\hline
				\multicolumn{2}{|c|}{$u,v$-Isolator path} & \multirow{2}{*}{\phantom{t}$\Delta m$\phantom{t}} & \multicolumn{1}{c|}{\multirow{2}{*}{\phantom{t}$\Delta k$\phantom{t}}} &  \multicolumn{1}{c|}{\multirow{2}{*}{\phantom{t}$\Delta l$\phantom{t}}} & \multicolumn{1}{c|}{\multirow{2}{*}{\phantom{t}$\Delta(m+4k-3l)$\phantom{t}}} & \multicolumn{1}{c|}{\multirow{2}{*}{\phantom{t}$p_I(G,P)$\phantom{t}}} \\ \cline{1-2}
				\multicolumn{1}{|c|}{\phantom{text}$u$\phantom{text}} & \multicolumn{1}{c|}{\phantom{text}$v$\phantom{text}}     &   & \multicolumn{1}{c|}{}   & \multicolumn{1}{c|}{}  & \multicolumn{1}{c|}{}     & \multicolumn{1}{c|}{}                \\ \hline
				\multicolumn{1}{|c|}{leaf}	  & \multicolumn{1}{c|}{leaf} & $-(r+1)$ & $-1$ & $-2$  & $-r+1$ & $0$ \\ \hline	\multicolumn{1}{|c|}{big}	  & \multicolumn{1}{c|}{leaf} & $-(r+1)$ & $0$  & $-1$  & $-r+2$ & $1$ \\ \hline
				\multicolumn{1}{|c|}{big}	  & \multicolumn{1}{c|}{big}  & $-(r+1)$  &$1$  & $0$   & $-r+3$ & $2$ \\ \hline
			\end{tabular} 
				\caption{The profit of deleting an Isolator path.} \label{table:isolator}
			\end{table}
			\item Let $G$ be a partially played graph which is a forest with $m$ edges, $k$ components, $l$ leaves and $q$ path components of length $1$.  Suppose $G-U$ is a forest with $m+\Delta m$ edges, $k+\Delta k$ components and $l+\Delta l$ leaves. Then the change in $m+4k-3l$ is as in Table~\ref{table:p2} and the profit $p_L(G,U)=\Delta(m+4k-3l)$ is equal to $q$.
			\begin{table}[H]
				\centering
				\begin{tabular}{|c|r|r|c|c|}
					\hline
					\multicolumn{1}{|c|}{\phantom{t}$\Delta m$\phantom{t}} & \multicolumn{1}{c|}{\phantom{t}$\Delta k$\phantom{t}} &  \multicolumn{1}{c|}{\phantom{t}$\Delta l$\phantom{t}}  & \multicolumn{1}{c|}{\phantom{t}$\Delta(m+4k-3l)$\phantom{t}}  & \multicolumn{1}{c|}{\phantom{t}$p_L(G,U)$\phantom{t}} \\ \hline
					$-q$ & $-q$ & $-2q$  & $q$  & $q$\\ \hline		
				\end{tabular}
				\caption{The profit of deleting $q$ path components of length $1$.} \label{table:p2}
			\end{table}
		\end{enumerate}
	\end{prop}
	
	\begin{proof} The calculation steps are shown in the tables. The profit $p_T(G,uv)\geq0$ since the term $+3$ in the definition of $p_T(G,uv)$ comes from $(-1)$ times the minimum value of $\Delta(m+4k-3l)$ in Table~\ref{table:toucher}. The profit $p_I(G,P)\geq0$ since the term $+(r-1)$ in the definition of $p_I(G,uv)$ comes from $(-1)$ times the minimum value of $\Delta(m+4k-3l)$ in Table~\ref{table:isolator}.	
	\end{proof}
	
	We are now ready to prove our main lemma which provides a lower bound for $\alpha(m,k,l)$ of the non-leaf Isolator-Toucher game on a forest.
	
	\begin{proof}[Proof of Lemma~\ref{lem:alpha}] We use induction on the number $m$ of edges in a forest. Let $F$ be a forest with $n$ vertices, $m$ edges, $k$ components, $l$ leaves, $a$ small vertices and $b$ big vertices. The base case is when all path components have length at most 2, all branches have length at most $2$, and all twigs have length at most $1$. In this case, we shall show that $\left\lfloor\frac{m+4k-3l+4}{5}\right\rfloor\leq0$, and so there is nothing to prove. Since $\sum_{v\in F}d(v)=2m=2(n-k)$, we have $l+2a+\sum_{d(v)\geq3}d(v)=2l+2a+2b-2k$. Then $l=\sum_{d(v)\geq3}d(v)-2b+2k$ and so $l\geq b+2k$. Since every edge in a non-path component is adjacent to a big vertex and every path component contains at most $2$ edges, it follows that
		\begin{center}
			$m\leq\displaystyle\displaystyle\sum_{d(v)\geq3}d(v)+2k=l+2b\leq 3l-4k$.
		\end{center} 
		Now, we suppose that there is either a path component of length at least $3$, a branch of length at least $3$, or a twig of length at least $2$.
		
		Isolator's strategy is to keep claiming consecutive edges, for as long as he can, to form an Isolator path. Therefore, he only plays within a path component, a branch, or a twig, say $P$. We label the edges of $P$ by $e_1, e_2,\dots, e_s$ respectively starting from a big edge (if exists). Note that we shall use this convention to label any path component, branch, or twig in this proof.  Assuming he has claimed the edges $e_t, e_{t+1},\dots,  e_{t+r}$, he then claims $e_{t-1}$ or $e_{t+r+1}$ if it is available, otherwise he stops. That is, he stops if ($t=1$ or $e_{t-1}$ is a Toucher edge) and ($t+r=s$ or $e_{t+r+1}$ is a Toucher edge). 
		
		Suppose Isolator stops with edges $e_t, e_{t+1},\dots,  e_{t+r}$. Then these edges form a path $Q$. So far, both players have claimed $r+1$ edges each since Isolator plays first, and the score is $r$ since Isolator creates an untouched vertex in every move except the first one. We  note that the case where Toucher has claimed only $r$ edges because all edges had been claimed, can be proved similarly. Let $G$ be the partially played graph at this step. If $f_1,\dots,f_{r+1}$ be the Toucher edges in $G$, then let $G_1=G\circleddash f_1\circleddash\dots\circleddash f_{r+1}$ be a forest with $m_1$ edges, $k_1$ components and $l_1$ leaves, let $G_2=G_1\circleddash Q$ be a forest with $m_2$ edges, $k_2$ components and $l_2$ leaves, and let $G_3=G_2-U$ be a forest with $m_3$ edges, $k_3$ components and $l_3$ leaves, where $U$ is the set of vertices of path components of length $1$ in $G_2$. 
		
		By Proposition~\ref{prop:equi}, the game on $G$ is equivalent to the game on $G_1$ which is equivalent to the game on $G_2$ with an extra score of $r$, and the score of the game on $G_2$ when both players play optimally is equal to that on $G_3$. Therefore, it follows that 
		\begingroup
        \allowdisplaybreaks
		\begin{align*}
			\alpha(m,k,l)&\geq r+\alpha(m_3,k_3,l_3)  \\	
			&\geq r+\left\lfloor\dfrac{m_3+4k_3-3l_3+4}{5}\right\rfloor   \hskip 5mm \text{(by the induction hypothesis)} \\	
			&= r+\left\lfloor\dfrac{m+4k-3l+4}{5}+\dfrac{\Delta_1(m+4k-3l)}{5}+\dfrac{\Delta_2(m+4k-3l)}{5} +\dfrac{\Delta_3(m+4k-3l)}{5} \right \rfloor \\ 
			&= r+\left\lfloor\dfrac{m+4k-3l+4}{5}+\dfrac{\sum_{i=0}^{r} (-3+p_T(G\circleddash f_1\circleddash\dots\circleddash f_{i},f_{i+1}))}{5}\right. \\
			&\hskip 38mm \left.+\dfrac{-r+1+p_I(G_1,Q)}{5}+\dfrac{p_{L}(G_2,U)}{5}\right\rfloor \\
			&\hskip 38mm \hskip 5mm \text{(by Proposition~\ref{prop:profit} since }Q\text{ is an Isolator path in }G_1)   \\
			&= r+\left\lfloor\dfrac{m+4k-3l+4}{5}+\dfrac{-3(r+1)+p_T}{5}+\dfrac{-r+1+p_I}{5}+\dfrac{p_{L}}{5}\right\rfloor  \\	
			&=\left\lfloor\dfrac{m+4k-3l+4}{5}+\dfrac{r+p_T+p_I+p_{L}-2}{5}\right\rfloor,
		\end{align*}
		where 	\begin{align*}
			\Delta_1(m+4k-3l)&=(m_1+4k_1-3l_1)-(m+4k-3l),  \\
			\Delta_2(m+4k-3l)&=(m_2+4k_2-3l_2)-(m_1+4k_1-3l_1),  \\
			\Delta_3(m+4k-3l)&=(m_3+4k_3-3l_3)-(m_2+4k_2-3l_2), \\ 
			p_T=p_T(G\circleddash f_1\circleddash\dots\circleddash &f_{i},f_{i+1}),\, p_I=p_I(G_1,Q)\text{ and }\, p_L=p_L(G_2,U).
		\end{align*}
		\endgroup
		
		Therefore, it suffices to show that $r+p_T+p_I+p_{L}\geq2$. Since every term in the sum $r+\sum p_T(G\circleddash f_1\circleddash\dots\circleddash f_{i},f_{i+1})+p_I+p_{L}$ is non-negative by Proposition~\ref{prop:profit}, we shall find a subset whose sum is at least 2. Recall that there is either a path component of length at least $3$, a branch of length at least $3$, or a twig of length at least $2$. The proof is divided into five cases.
		
		\textbf{Case 1.} There is a path component of length $3$. 
		
		Isolator claims the edge $e_2$ in his first move. If Toucher claims the leaf edge $e_1$ or $e_3$ in some move, then $p_T\geq2$ by Proposition~\ref{prop:profit}. Otherwise, Isolator claims the edges $e_1$ and $e_3$, hence $r=2$.

		\textbf{Case 2.} There is a path component of length at least $4$.	
		
		Isolator claims the edge $e_3$ in his first move. If Toucher claims the leaf edge $e_1$ in some move, then $p_T\geq2$ by Proposition~\ref{prop:profit}.  If Toucher claims the edge $e_2$ in some move (but not $e_1$), then $G_2$ has a path component $e_1$ of length $1$ and so $p_{L}\geq1$ by Proposition~\ref{prop:profit}. Clearly, $r\geq1$, hence it follows that $r+p_{L}\geq2$. Otherwise, Isolator claims the edges $e_1$ and $e_2$, hence $r\geq2$.
		
		\textbf{Case 3.} There is a branch of length at least $3$.	
		
		Isolator claims the edge $e_2$ in his first move. If Toucher claims the big edge $e_1$ in some move, then  $p_T\geq1$ by Proposition~\ref{prop:profit}. Clearly, $r\geq1$, hence it follows that $r+p_T\geq2$. If  Toucher claims the edge $e_3$ in some move, then $p_I\geq1$ by Proposition~\ref{prop:profit} since Isolator claims the big edge $e_1$. Clearly, $r=1$, hence it follows that $r+p_I\geq2$. Otherwise, Isolator claims the edges $e_1$ and $e_3$, hence $r\geq2$.

		\textbf{Case 4.} There is a twig of length $2$.
		
		Isolator claims the edge $e_1$ in his first move. If Toucher claims the leaf edge $e_2$ in some move, then $p_T\geq2$ by Proposition~\ref{prop:profit}. Otherwise, Isolator claims the edge $e_2$, hence $p_I\geq1$ by Proposition~\ref{prop:profit} since Isolator claims the big edge $e_1$. Clearly, $r=1$, hence it follows that $r+p_I\geq2$.

		\textbf{Case 5.} There is a twig of length at least $3$.
		
		Isolator claims the edge $e_2$ in his first move. If Toucher claims the big edge $e_1$ in some move, then $p_T\geq 1$ by Proposition~\ref{prop:profit}. Clearly, $r\geq1$, hence it follows that $r+p_T\geq2$. If Toucher claims the edge $e_3$ in some move, then $p_I\geq1$ by Proposition~\ref{prop:profit} since Isolator claims the big edge $e_1$.  Clearly, $r=1$,  it follows that $r+p_I\geq2$. Otherwise, Isolator claims the edges $e_1$ and $e_3$, hence $r\geq2$.
		
		This completes the proof of Lemma~\ref{lem:alpha}. 
	\end{proof}

	We now prove Theorem~\ref{thm:uT} which improves the lower bound for $u(T)$ of the Toucher-Isolator game, by applying the result on the non-leaf Isolator-Toucher game in Lemma~\ref{lem:alpha}.

	\begin{proof}[Proof of Theorem~\ref{thm:uT}] Let $T$ be a tree with $m\geq2$ edges, $k$ components and $l$ leaves. We shall show that 
		\begin{center}
			$u(T)\geq\left\lfloor\dfrac{m+4}{5}\right\rfloor$.
		\end{center}
		
		For a partially played graph $G$, a \emph{meta-leaf} in $G$ is a leaf in the graph obtained from $G$ by deleting all Isolator edges, and a \emph{meta-leaf edge} in $G$ is an edge incident to a meta-leaf in $G$. 
		
		Isolator's strategy is to keep claiming an edge which produces a new untouched vertex in every move i.e., he claims a meta-leaf edge in the current partially played graph if it is available, otherwise he stops (see Figure~\ref{figure:thm}: left). That is, he stops when all meta-leaf edges are Toucher edges. We note that he always obtains a score of one in every move because if he claims the edge $uv$ where $u$ is a meta-leaf, then all already played edges incident to $u$ are Isolator edge, and so $u$ becomes untouched. If the process stops after Isolator's move, i.e. all edges have been claimed by both players, then Isolator obtains a score of $\left\lfloor\frac{m}{2}\right\rfloor\geq\left\lfloor\frac{m+4}{5}\right\rfloor$, as required. Therefore, we may assume that the process stops after Toucher's move, and in particular, $m\geq3$.
			
		Suppose Isolator stops after $r$ moves. Let $G$ be the partially played graph at this step. Then $G$ has $r+1$ Toucher edges and $r$ Isolator edges since Toucher plays first. Let $H$ be the Isolator subrgaph of $G$ formed by all Isolator edges, and let $G_1=G\circleddash H$ be a forest with $m_1$ edges, $k_1$ components and $l_1$ leaves (see Figure~\ref{figure:thm}: middle). Since Isolator claimed only meta-leaf edges and all meta-leaf edges in $G$ are Toucher edges, $G_1$ is a tree all of whose leaves are touched. By $m\geq3$, each leaf of $G_1$ is incident to a distinct Toucher edge, and so $r+1\geq l_1$. Let $f_1,\dots,f_{r+1}$ be the Toucher edges in $G$, and let $G_2=G_1\circleddash f_1\circleddash\dots\circleddash f_{r+1}$ be the forest with $m_2$ edges, $k_2$ components and $l_2$ leaves (see Figure~\ref{figure:thm}: right). 
		
		By Proposition~\ref{prop:equi} and the fact that the leaves in $G_1$ are touched, the Toucher-Isolator game on $G$ where Isolator plays first is equivalent to the non-leaf Isolator-Toucher game on $G_1$ which is equivalent to the non-leaf Isolator-Toucher game on $G_2$ with an extra score of $r$. Therefore, it follows that
        \allowdisplaybreaks
		\begin{align*}
			u(T)&\geq r+\alpha(m_2,k_2,l_2)\\
			&\geq r+\left\lfloor\dfrac{m+4k-3l+4}{5}+\dfrac{\Delta_1(m+4k-3l)}{5}+\dfrac{\Delta_2(m+4k-3l)}{5}\right\rfloor  \hskip 5mm \text{(by Lemma }\ref{lem:alpha})\\
			&= r+\left\lfloor\dfrac{m+4k-3l+4}{5}+\dfrac{(m_1-m)+4(k_1-k)-3(l_1-l)}{5}\right. \\
			&\phantom{= r. } \left.+\dfrac{\sum_{i=0}^{r} (-3+p_T(G_1\circleddash f_1\circleddash\dots\circleddash f_{i},f_{i+1}))}{5}\right\rfloor   \\
			&\geq r+\left\lfloor\dfrac{m+4k-3l+4}{5}+\dfrac{(-r)+4(0)-3(l_1-l)}{5}\right. \\
			&\phantom{= r. } \left.+\dfrac{-3(r+1)+2l_1}{5}\right\rfloor \hskip 5mm   \text{(by Proposition~\ref{prop:profit} since }G_1\text{ has }l_1\text{ leaf edges}) \\
			&=\left\lfloor\dfrac{m+4k+r-l_1+1}{5}\right\rfloor \\
			&\geq\left\lfloor\dfrac{m+4k}{5}\right\rfloor  \hskip 5mm  \text{($r-l_1-1\geq0$)} \\
			&=\left\lfloor\dfrac{m+4}{5}\right\rfloor, \hskip 5mm  \text{($k=1$)}
		\end{align*}	
		where 	\begin{align*}	\Delta_1(m+4k-3l)&=(m_1+4k_1-3l_1)-(m+4k-3l),  \\
		\Delta_2(m+4k-3l)&=(m_2+4k_2-3l_2)-(m_1+4k_1-3l_1).  \qedhere 
		\end{align*} 
	
	\end{proof}

	\section{Concluding Remarks}\label{section-conclusion}
	
	As a result of Theorem~\ref{thm:uT}, for any tree $T$ with $n\geq3$ vertices,
	\begin{center}
		$u(P_n)\leq u(T)\leq u(S_n)$,
	\end{center}
	where $S_n$ is a star with $n$ vertices. Moreover, Theorem~\ref{thm:uT} implies that, for a forest with $k$ trees, $u(F)\geq\sum_{i=1}^{k}\left\lfloor\frac{n_i+3}{5}\right\rfloor$, where $n_i$ is the number of vertices of the $i^{\text{th}}$ tree in $F$ because, in each move, Isolator can play optimally on the tree Toucher just played. However, the lower bound of $\left\lfloor\frac{n+3k}{5}\right\rfloor$ is not possible because for example, $u(kP_3)=k$  where $kP_3$ is the disjoint union of $k$ copies of $P_3$. Many interesting questions about the Toucher-Isolator game are still open (see~\cite{MR4025410}). 
	For example, find a $3$-regular graph $G$ with $n$ vertices that maximizes $u(G)$. Dowden, Kang, Mikala\v{c}ki and Stojakovi\'{c} showed that the largest proportion of untouched vertices for a $3$-regular graph is between $\frac{1}{24}$ and $\frac{1}{8}$.
	
\section*{Acknowledgment}	

The first author is grateful for financial support from the Science Achievement Scholarship of Thailand. 

	\bibliographystyle{siam} 
	\bibliography{TIgamebib}
	
\end{document}